\documentclass{file}
\usepackage{pat}
\usepackage{paralist}
\usepackage{dsfont}  
\usepackage[utf8x]{inputenc}
\usepackage{skull}
\usepackage{paralist}
\usepackage{centernot}
\usepackage{mathtools}
\usepackage{thmtools}
\usepackage{hyperref}
\usepackage[table]{xcolor}
\usepackage{stmaryrd,tikz}
\usetikzlibrary{calc}
\usepackage{graphicx}
\usepackage{booktabs}
\usepackage{listings}
\usepackage{pifont}%
\usepackage[normalem]{ulem}
\usepackage{cancel,float}
\usepackage{comment}

\usepackage{todonotes}
\usepackage{mathdots}
\usepackage{bm}
\usepackage{etoolbox}
\definecolor{maroon}{rgb}{0.5, 0.0, 0.0}
\definecolor{darkblue}{rgb}{0.0, 0.0, 0.55}

\usepackage[longnamesfirst,numbers,sort&compress]{natbib}

\usepackage[mathlines]{lineno}
\newcommand*\patchAmsMathEnvironmentForLineno[1]{%
 \expandafter\let\csname old#1\expandafter\endcsname\csname #1\endcsname
 \expandafter\let\csname oldend#1\expandafter\endcsname\csname end#1\endcsname
 \renewenvironment{#1}%
    {\linenomath\csname old#1\endcsname}%
    {\csname oldend#1\endcsname\endlinenomath}}%
\newcommand*\patchBothAmsMathEnvironmentsForLineno[1]{%
 \patchAmsMathEnvironmentForLineno{#1}%
 \patchAmsMathEnvironmentForLineno{#1*}}%
\AtBeginDocument{%
\patchBothAmsMathEnvironmentsForLineno{equation}%
\patchBothAmsMathEnvironmentsForLineno{align}%
\patchBothAmsMathEnvironmentsForLineno{flalign}%
\patchBothAmsMathEnvironmentsForLineno{alignat}%
\patchBothAmsMathEnvironmentsForLineno{gather}%
\patchBothAmsMathEnvironmentsForLineno{multline}%
}

\definecolor{brightmaroon}{rgb}{0.76, 0.13, 0.28}
\definecolor{linkblue}{rgb}{0, 0.337, 0.227}
\newcommand{\rank}{\mathop{\mathsf{Rank}}}
\newcommand{\spec}{\mathop{\mathsf{Spec}}}
\newcommand{\tr}{\mathop{\mathsf{Trace}}}
\newcommand{\aut}{\mathop{\mathsf{Aut}}}

\newrobustcmd{\onesub}{\mathord{\includegraphics{one-sub}}}
\newrobustcmd{\leftup}{\mathord{\includegraphics{left-up}}}
\makeatletter
\newcommand{\xMapsto}[2][]{\ext@arrow 0599{\Mapstofill@}{#1}{#2}}
\def\Mapstofill@{\arrowfill@{\Mapstochar\Relbar}\Relbar\Rightarrow}
\makeatother

\title{\MakeUppercase{On Matrix Product Factorization in Association Schemes}}
\author{Allen W. Herman
\thanks{Department of Mathematics and Statistics, University of Regina},\quad
Bobby Miraftab
\thanks{School of Computer Science,
Carleton University}
}

\date{}

\begin{document}

\maketitle

\begin{abstract}
We study matrix product factorizations (MPFs) in symmetric association schemes:
identities $A_SA_T=A_U$ where $A_S,A_T,A_U$ are loopless unions of basic
relations and the ordinary matrix product is again a $0$-$1$ adjacency
matrix.  We give equivalent structural and spectral criteria for MPFs, derive
valency and rank restrictions, and analyze several standard families.  For
$2$-class schemes, the only nontrivial loopless MPF comes from the scheme of
the $5$-cycle.  For $P$-polynomial schemes, the distance-regular recurrence
gives strong restrictions on products $A_1A_i$.  We also prove a universal
pentagon theorem for the case $A_SA_T=J-I$, and show that extremal rank forces
all non-zero eigenvalues of $A_U$ to be $\pm k(U)$, hence gives bipartiteness.
Finally, in Hamming schemes we obtain rank obstructions and classify MPFs of
the form $A_1A_T=A_U$: in $H(d,2)$, for $d\ge2$, the only non-zero loopless
example is $A_1A_d=A_{d-1}$, which is trivial since $A_d$ has valency $1$;
for $q>2$, no non-zero example occurs.
\end{abstract}

\section{Introduction}

Association schemes are a standard framework for organizing families of highly
regular graphs on a common vertex set.  
In this paper we study a simple closure phenomenon inside the Bose--Mesner algebra.  
A loopless \defin{graph in the scheme} is a union of basic relations, with adjacency matrix
\[
A_S=\sum_{i\in S}A_i,\qquad S\subseteq\{1,\dots,d\}.
\]
Given two such graphs $A_S$ and $A_T$, their product $A_SA_T$ counts $S\!\to\!T$
two--step walks.  Typically this product has entries larger than $1$, so it is
not the adjacency matrix of a simple graph.  We isolate the exceptional case:

We say that $A_SA_T=A_U$ is a (loopless) \defin{matrix-product factorization (MPF)} if $S,T,U\subseteq\{1,\dots,d\}$ and the matrix product is again a $\{0,1\}$ adjacency matrix of a loopless graph in
the scheme.

Equivalently, for every ordered pair $(x,y)$ there is \emph{at most one}
intermediate vertex $z$ that realizes an $S$--step followed by a $T$--step from
$x$ to $y$, and there are no such $2$--walks from $x$ back to $x$ (no loops).
This viewpoint makes MPFs a natural ``uniqueness'' condition on two--step walks in a colored complete graph.
This approach generalizes the previous frame work on MPF of graphs, see
\cite{maghsoudi2023matrix,miraftab2025factorability,prime,herman2025matrix}.

Our first goal is to give usable criteria for deciding when an MPF occurs in our setting.
We provide two equivalent tests.  One is combinatorial, in terms of the
intersection numbers~\cref{prop:structure-test}).
The other is
spectral, in terms of the first eigenmatrix $P$~\Cref{prop:spectral-test}),
which turns MPF existence into a discrete subset--sum constraint on the columns
of $P$ (~\Cref{rem:subset-sum}).  
We also provide general consequences for
valencies, such as multiplicativity $k(U)=k(S)k(T)$ and a universal degree--sum
bound (\Cref{cor:valency-multiplicativity}).

Our second goal is to understand MPFs in important families of schemes.
For $2$--class schemes (equivalently, strongly regular graphs and their complements),
we show that the only nontrivial loopless MPF forces the scheme to be the one
coming from the $5$--cycle.  For $P$--polynomial schemes (distance--regular graphs),
the three--term recurrence for distance matrices imposes strong restrictions on
when products like $A_1A_i$ can be $\{0,1\}$--matrices; we derive explicit
constraints in several cases (\Cref{sec:DRG}).

Our main new general results are two rigidity statements.  
\Cref{thm:universal-pentagon}
shows that if the nontrivial relations split as $S\sqcup T$(disjoint union) and
$A_SA_T=J-I$, then necessarily $v=5$ and $A_S,A_T$ are complementary $5$--cycles
(i.e., the corresponding fusion scheme is the $2$--class scheme of $C_5$).
\Cref{thm:rank-mpf} introduces a rank method: any MPF satisfies
\[
\rank (A_U)\le \min(\rank (A_S),\rank (A_T)),
\]
and, when $U\ne\emptyset$, the output gives the lower bound $\rank(A_S),\rank(A_T)\ge \frac{v}{k(U)}$.
We then analyze the extremal case
$\rank(A_U)=\frac{v}{k(U)}$
and show that every nonzero eigenvalue of $A_U$ is equal to
$\pm k(U)$, which implies bipartiteness
(~\Cref{thm:DRG-extremal-rank}).

Finally, we apply these tools to concrete schemes.  In Hamming schemes we obtain
rank obstructions for MPFs with a single distance class as output.  In the
binary case $H(d,2)$, for $d\ge2$, the only nonzero loopless MPF of the form
$A_1A_T=A_U$ is $A_1A_d=A_{d-1}$,
which is trivial since $A_d$ has valency $1$.  For $q>2$, no nonzero
loopless MPF of this form occurs.

\medskip
\noindent\textbf{Organization.}
\Cref{sec:prel} recalls basic notation and proves the structural and spectral MPF
tests.  
\Cref{sec:srg} and \Cref{sec:DRG} treat the strongly regular and distance--regular
settings.  
\Cref{sec:hom} extends the MPF notion to homogeneous coherent configurations.  \Cref{sec:rankbounds} introduces the rank method, \Cref{sec:hamm} develops its applications to Hamming schemes and related examples.  
\Cref{sec:cyclotomic} relates MPFs in translation and cyclotomic schemes to
group-ring factorizations and $\lambda$-fold near-factorizations.
The universal pentagon theorem is proved in \Cref{sec:universal-pentagon}.

\section{Preliminaries}\label{sec:prel}
A \defin{$d$-class association scheme} on a finite set $X$ is a collection of
relations $\mathcal{R} = \{R_0, R_1, \dots, R_d \}$ on $X$ such that:
\begin{enumerate}
    \item $R_0 = \{ (x,x) : x \in X \}$ is the diagonal relation.
    \item $\mathcal{R}$ is a partition of $X \times X$, i.e.,
    $X \times X = R_0 \,\sqcup \, R_1 \,\sqcup \, \cdots \,\sqcup \, R_d$.
    \item $\mathcal{R}$ is transpose-invariant: $\forall R_i \in \mathcal{R}, R_{i'} = \{ (x,y) \in X \times X : (y,x) \in R_i \} \in \mathcal{R}$.
    
    \item For any $i,j,k \in \{0,\dots,d\}$ and any $(x,y) \in R_k$, the number
    \[
        p_{ij}^k = \left| \{ z \in X : (x,z) \in R_i \text{ and } (z,y) \in R_j \} \right|
    \]
    is constant (i.e., independent of the choice of $(x,y) \in R_k$). These constants
    $p_{ij}^k$ are called the \defin{intersection numbers} of the scheme.
    \item $\forall i,j,k \in \{0,\dots,d\}$, $p_{ij}^k = p_{ji}^k$. 
\end{enumerate}
The relations $R_0,\dots,R_d$ are called the \defin{associate classes} of the scheme.

Let $\mathcal{R}=\{R_0,R_1,\dots,R_d\}$ be an association scheme on the finite vertex set $X$.
Denote by $A_0,\dots,A_d$ the adjacency matrices of the relations $R_0,\dots,R_d$, with $A_0=I$ and $\sum_{i=0}^d A_i=J$.  For each $A_i$, the matrix transpose $A_i^{\top} = A_{i'}$ is also in the set $\{A_0,A_1,\dots,A_d\}$.  The association scheme is said to be \defin{symmetric} when all of its adjacency matrices $A_i$ are symmetric.  Since 
\[
A_iA_j=\sum_{k=0}^d p_{ij}^k\,A_k\qquad(0\le i,j\le d).
\]
the \defin{Bose–Mesner algebra} $\mathbb{C}[\mathfrak{X}]=\mathrm{span}\{A_0,\dots,A_d\}$ is a commutative semisimple algebra under ordinary matrix multiplication, whose structure constants in the basis $\{A_0,A_1,\dots,A_d\}$ are precisely the intersection numbers.  For all $i=0,1,\dots,d$, the intersection number $k_i:=p_{ii'}^0$ corresponds to the eigenvalue of $A_i$ whose eigenvector is the all $1$'s vector $\mathbf{1}$, we will refer to this positive integer as the \defin{valency} of $A_i$.   
For each basic relation $A_i$, let $k_i:=p_{ii'}^0$ denote its valency, i.e.,
$A_i\mathbf{1}=k_i\mathbf{1}$.
For an index set $S\subseteq\{1,\dots,d\}$ define the \defin{valency (degree)} of the union
\[
A_S:=\sum_{i\in S}A_i
\qquad\text{by}\qquad
k(S):=\sum_{i\in S}k_i,
\]
equivalently $A_S\mathbf{1}=k(S)\mathbf{1}$.


Let $\mathcal{R}$ be a symmetric $d$-class association scheme on the vertex set $X$.  A (loopless) \defin{graph} in the scheme will be a simple graph on $X$ whose adjacency matrix is a $01$-matrix of the form
\[
A_S \;:=\; \sum_{i\in S} A_i,
\qquad S\subseteq\{1,2,\dots,d\},
\]
i.e., a union of basic relations avoiding $R_0$ (to exclude loops).  
Likewise, for another index set $T\subseteq\{1,\dots,d\}$ we set
\[
A_T \;:=\; \sum_{j\in T} A_j.
\]
We note that $S$ records \emph{which relation indices are allowed on the first step} of a walk, and $T$ records \emph{which relation indices are allowed on the second step}. Thus $S$ and $T$ are simply subsets of indices $\{1,\dots,d\}$ labelling unions of relations.

\begin{defn}
We say that $A_SA_T=A_U$ is a \defin{matrix product factorization} or for short \defin{MPF} if $U\subseteq\{1,\dots,d\}$ and
\[
A_SA_T \;=\; \sum_{k\in U}A_k,
\]
i.e., the two–step walk ``through $S$ then $T$'' again yields the adjacency matrix of a (loopless) graph in the scheme.
\end{defn}
It is not hard to see that
\[
A_SA_T
= \Bigl(\sum_{i\in S}A_i\Bigr)\Bigl(\sum_{j\in T}A_j\Bigr)
= \sum_{i\in S}\sum_{j\in T} A_iA_j
= \sum_{i\in S}\sum_{j\in T}\sum_{k=0}^d p_{ij}^k\,A_k
= \sum_{k=0}^d c_k\,A_k,
\]
where
\begin{equation}\label{eq:coefficients}
c_k \;:=\; \sum_{i\in S}\sum_{j\in T} p_{ij}^k \qquad (0\le k\le d).
\end{equation}
\begin{prop}
\label{prop:structure-test}
For $S,T\subseteq\{1,\dots,d\}$ the following are equivalent:
\begin{enumerate}[(i)]
\item There exists $U\subseteq\{1,\dots,d\}$ with $A_SA_T=A_U$ (i.e., an MPF with no loops).
\item The coefficients in (\ref{eq:coefficients})  satisfy
$c_k\in\{0,1\}\ \text{for all}\ k=1,\dots,d$,
and $c_0=0$.
In that case $U=\{\,k\in\{1,\dots,d\}: c_k=1\,\}$.
\end{enumerate}
\end{prop}

\begin{proof}
$(i)\Rightarrow(ii)$: If $A_SA_T=A_U=\sum_{k\in U}A_k$ with $0\notin U$, then comparing the $A_k$–coefficients (the $A_k$'s are a basis) shows $c_k=1$ for $k\in U$ and $c_k=0$ otherwise; in particular $c_0=0$.

$(ii)\Rightarrow(i)$: If $c_k\in\{0,1\}$ for all $k$ and $c_0=0$, then
\[
A_SA_T=\sum_{k=1}^d c_k\,A_k=\sum_{k\in U}A_k
\]
with $U=\{k\ge 1: c_k=1\}$, which is the adjacency matrix of a union of relations, i.e., a loopless graph in the scheme.
\end{proof}

\begin{rem}
\label{rem:unique-2step}
By the defining property of $p_{ij}^k$, the number of $S\!\to\!T$ two–step walks from $x$ to $y$ depends only on the relation class $R_k$ containing $(x,y)$ and equals $c_k=\sum_{i\in S,j\in T}p_{ij}^k$.
Thus the condition $c_k\le 1$ for all $k$ says:
\emph{for every ordered pair $(x,y)$ there is \emph{at most one} intermediate $z$ and choices $i\in S$, $j\in T$ with $(x,z)\in R_i$, $(z,y)\in R_j$.}
The additional requirement $c_0=0$ forbids loops.
\end{rem}

Let $E_0,\dots,E_d$ be the primitive idempotents of the Bose–Mesner algebra, with $E_0 = \frac{1}{|X|} J$ corresponding to the valency map.  Let $P=(P_{h i})$ be the first eigenmatrix, so that
\[
A_i E_h \;=\; P_{h i}\,E_h \qquad(0\le h,i\le d),\quad\text{and}\quad
A_i \;=\; \sum_{h=0}^d P_{h i}\,E_h.
\]
For an index set $S\subseteq\{0,\dots,d\}$ define
\[
\lambda_h(S) \;:=\; \sum_{i\in S} P_{h i},
\qquad\text{so that}\qquad
A_S E_h \;=\; \lambda_h(S)\,E_h,
\]
and similarly for $T$ and any $U$.

\begin{prop}
\label{prop:spectral-test}
Let $S,T,U\subseteq\{1,\dots,d\}$.  Then the following are equivalent:
\begin{enumerate}[(i)]
\item $A_SA_T=A_U$.
\item For all $h=0,1,\dots,d$, we have
$\lambda_h(S)\lambda_h(T)=\lambda_h(U)$.
\end{enumerate}
Equivalently,
\[
A_SA_T = A_U
\quad\Longleftrightarrow\quad
\bigl(\forall h\bigr)\ 
\sum_{i\in S}P_{h i}\cdot \sum_{j\in T}P_{h j}
=
\sum_{k\in U}P_{h k}.
\]
Consequently, for fixed $S,T$, there exists a loopless MPF
$A_SA_T=A_U$ if and only if there exists
$U\subseteq\{1,\dots,d\}$ satisfying these equivalent conditions.
\end{prop}

\begin{proof}
Since $A_SE_h=\lambda_h(S)E_h$ and $A_TE_h=\lambda_h(T)E_h$, we have
\[
A_SA_T E_h
=
A_S\bigl(\lambda_h(T)E_h\bigr)
=
\lambda_h(T)A_SE_h
=
\lambda_h(S)\lambda_h(T)E_h.
\]
If $A_SA_T=A_U$, then $A_UE_h=\lambda_h(U)E_h$, so comparison on each
primitive idempotent gives
$\lambda_h(S)\lambda_h(T)=\lambda_h(U)$, where 
$(0\le h\le d)$.

Conversely, if these equalities hold for all $h$, then
\[
A_SA_T
=
\sum_{h=0}^d \lambda_h(S)\lambda_h(T)E_h
=
\sum_{h=0}^d \lambda_h(U)E_h
=
A_U,
\]
using the primitive idempotent basis of the Bose--Mesner algebra.
\end{proof}

Note that the eigenvalue of each $A_i$ corresponding to $E_0 = |X|^{-1}J$ is the valency $k_i$. For an index set $S\subseteq\{1,\dots,d\}$ write
\[
A_S:=\sum_{i\in S}A_i,\qquad k(S):=\text{the common row sum of }A_S=\sum_{i\in S}k_i = \lambda_0(A_S).
\]

\begin{cor}\label{cor:valency-multiplicativity}
If $A_SA_T=A_U$ for some $S,T,U\subseteq\{1,\dots,d\}$ (with $0\notin S\cup T\cup U$), then
$k(U)=k(S)\,k(T)$.
Furthermore, if the association scheme has $v$ vertices, then $k(S)+k(T) \le v-1$.  
\end{cor}

\begin{proof} 
That $A_S A_T = A_U$ implies $k(S) \, k(T) = k(U)$ is immediate from the $\lambda_0$ case of \Cref{prop:spectral-test}.  
Since $A_U$ is the adjacency matrix of a $k(U)$-regular graph on $v$ vertices, it can be at most $(v-1)$-regular, with equality if and only if $A_U = J-I$. 

For the second statement, suppose $1 \le k(S) \le k(T)$.  If $1 < k(S), k(T)$, then $k(S) + k(T) \le k(S)\, k(T)$, so we are done.  If $1 = k(S) \le k(T)$, then $A_S$ is the adjacency matrix of a $1$-regular graph (i.e. a perfect matching) on $v$ vertices, so $v$ must be even, and $A_T$ is the adjacency matrix of a $k(T)$-regular graph on these same $v$ vertices.  These two graphs cannot have an edge in common, because if they did the graph whose adjacency matrix is $A_U$ would have a loop. So $A_T $ is the adjacency matrix of a graph which is at most $(v-2)$-regular, and we have $k(S) + k(T) \le v-1$, as required. 
\end{proof}
So by \Cref{prop:spectral-test}, we have that 
$\lambda_0(A_S)\lambda_0(A_T) = \lambda_0(A_U)$ gives a necessary condition on valencies for any matrix product factorization in an association scheme:  
\begin{equation}\label{eq:valencies}
A_S A_T = A_U \implies \big(\sum_{i \in S} k_i \big) \big(\sum_{j \in T} k_j\big) = \big(\sum_{\ell \in U} k_{\ell}\big).
\end{equation}

\begin{rem}
\label{rem:subset-sum}
Writing $u=\mathbf{1}_U\in\{0,1\}^{d+1}$ and similarly $s=\mathbf{1}_S$, $t=\mathbf{1}_T$, the equalities in \Cref{prop:spectral-test} say
\[
P\,u \;=\; (P\,s)\odot(P\,t),
\]
where $\odot$ denotes entrywise (Hadamard) product.
Thus deciding whether there \emph{exists} $U$ with $A_SA_T=A_U$ is a discrete \emph{subset–sum factorization} problem on the rows of $P$: the vector $(P\,s)\odot(P\,t)$ must itself be a \emph{subset sum of columns} of $P$ (with coefficients in $\{0,1\}$), and to be loopless we must have $0\notin U$.
\end{rem}

\begin{rem}
If you allow loops, you may include $0\in S$ or $0\in T$, and $U$ may contain $0$.
For loopless MPFs, assume $S,T\subseteq\{1,\dots,d\}$ and enforce either of the equivalent conditions
\[
c_0=\sum_{i\in S}\sum_{j\in T}p_{ij}^0=0
\qquad\Longleftrightarrow\qquad
0\notin U.
\]
Note that $A_0=I$ corresponds to the constant column $P_{\bullet,0}\equiv 1$ in the eigenmatrix.
\end{rem}

\begin{rem}
If $A_1 \ne A_0$ is an adjacency matrix in a symmetric association scheme with valency $1$, then $A_1$ is a symmetric permutation matrix, so it is diagonalizable with eigenvalues $1$ and $-1$, both of which occur. Therefore, $A_1^2 = I = A_0$ and $A_1J = J$, so $\sum_{j > 1} A_j = J - I - A_1$ satisfies 
$$A_1(J-I-A_1) = (J-I-A_1).$$   
\end{rem}

\begin{defn}
We will say that a matrix product factorization in a symmetric association scheme is \defin{trivial} when one of the factors has valency $1$.  
\end{defn}


\section{SRGs}\label{sec:srg}

Suppose $A_1$ is the adjacency matrix of a strongly-regular graph $X$.  We can assume the association scheme $S$ corresponding to this strongly-regular graph has adjacency matrices $A_0=I$, $A_1$, and $A_2$, where $A_2$ is the adjacency matrix of the complementary strongly-regular graph.  The eigenvalues of $A_1$ can be assumed to be $k \ge r \ge 0 > -1 \ge s$, and the eigenvalues of $A_2$ are $\ell$, $-1-r$, and $-1-s$, with the same eigenspaces, respectively.  

The only non-zero graphs with no loops in this $2$-class scheme have adjacency
matrices $A_1, A_2, A_1+A_2=J-I$.
Suppose $XY=Z$ is a loopless MPF of graphs in this scheme, with
$X,Y,Z\in\{A_1,A_2,A_1+A_2\}$.  We first rule out all products except
$A_1A_2$.

Indeed, $A_1^2$ has diagonal coefficient $k$, and $A_2^2$ has diagonal
coefficient $\ell$.  Hence neither $A_1^2$ nor $A_2^2$ can be loopless.
Also, if one of the factors is $A_1+A_2=J-I$, then
\[
(A_1+A_2)A_1=(J-I)A_1=kJ-A_1
\]
has diagonal coefficient $k$, and $(A_1+A_2)A_2=(J-I)A_2=\ell J-A_2$
has diagonal coefficient $\ell$.  Finally,
$(A_1+A_2)^2=(J-I)^2=(v-1)I+(v-2)(A_1+A_2)$
also has non-zero diagonal coefficient.  Therefore, the only possible loopless
MPFs in this $2$-class scheme are of the form
$A_1A_2=A_1$, $A_1A_2=A_2$, or $A_1A_2=A_1+A_2$.

The first two cases are trivial.  If $A_1A_2=A_1$, then comparing valencies
gives $k\ell=k$,
so $\ell=1$.  Thus $A_2$ has valency $1$.  Similarly, if
$A_1A_2=A_2$, then $k\ell=\ell$,
so $k=1$, and $A_1$ has valency $1$.  Hence both cases are trivial in
the sense of our definition.

It remains to consider the only possible nontrivial case, $A_1A_2=A_1+A_2$.
We now show that this forces the scheme to be the $2$-class scheme of the
$5$-cycle.

\begin{thm} Suppose $A_1$, $A_2$ are the adjacency matrices of complementary strongly-regular graphs. If $A_1 A_2 = A_1 + A_2$, 
then $A_1$ is the adjacency matrix of a $5$-cycle.  
\end{thm} 
\begin{proof} Considering valencies, $A_1A_2 = A_1 + A_2$ implies $k \ell = k + \ell$ for positive integers $k$ and $\ell$, so the only possibility is $k=\ell=2$.  Projecting to the $r$- and $s$-eigenspaces of $A_1$ we see that $r(-1-r) = -1 = s(-1-s)$, so $r$ and $s$ are distinct roots of the polynomial $x^2+x-1$.  Thus $r,s = \frac{-1 \pm \sqrt{5}}{2}$.  The only strongly-regular graph with these eigenvalues is the $5$-cycle (and its complement).  
\end{proof}

\section{DRGs}\label{sec:DRG}

Let $XY=Z$ be a matrix factorization of graphs in an association scheme.  Suppose $X$ is a distance-regular graph, and $Y$ and $Z$ are graphs in the association scheme generated by $X$.  Then $Y$ and $Z$ are unions of distance graphs in the association scheme.  
The intersection matrix $B_1$ of $X$ is determined by its intersection array 
$$[b_0,b_1,\dots,b_{d-1};1,c_2,\dots,c_d],$$ 
in that $B_1$ is tridiagonal, with all row sums equal to $b_0$, the valency of $X$.  Furthermore, the $b_i$ are positive and decreasing, the $c_i$ are positive and increasing.  A standard reference for these properties is \cite{BCN}. 

Suppose $Y$ is a distance matrix for $X$, so its adjacency matrix is $A_i$.  If $A_1A_i$ is a graph in the association scheme, it has to be that either (i) $2 \le i<d$ and $A_1A_i = A_{i-1} + A_{i+1}$; (ii) $2 \le i<d$ and $A_1A_i = A_{i-1} + A_i + A_{i+1}$; (iii) $A_1 A_d = A_{d-1}+A_d$, or (iv) $A_1 A_d = A_{d-1}$.
\begin{thm} 
The following holds:
\begin{enumerate}[(i)]
    \item Suppose $2\le i<d$ and $A_1A_i=A_{i-1}+A_{i+1}$.
    Then $b_0=2$ and $X$ is an $n$-cycle.

    \item Suppose $2\le i<d$ and $A_1A_i=A_{i-1}+A_i+A_{i+1}$.
    Then $b_0=3$, $b_j=1$ for $j=i-1,\dots,d-1$, and
    $c_k=1$ for $k=1,\dots,i+1$.
    \item Suppose $i=d$ and $A_1A_d = A_{d-1}+A_d$.  Then $b_{d-1} = 1$ and $c_d = b_0-1$. 
    \item Suppose $i=d$ and $A_1A_d=A_{d-1}$.  Then $b_{d-1}=1,  a_d=0, c_d=b_0$.
    Equivalently, this case occurs exactly when $b_{d-1}=1$ and $a_d=0$.
\end{enumerate}

\end{thm}

\begin{proof} 
(i) Assume $2\le i<d$ and $A_1A_i=A_{i-1}+A_{i+1}$.
Then $b_{i-1}=1, a_i=0, c_{i+1}=1$.
Since the $b_j$'s are positive and non-increasing, $b_i=1$.  Since the
$c_j$'s are positive and non-decreasing, $c_i=1$.  Hence $b_0=b_i+a_i+c_i=1+0+1=2$.
The only distance-regular graphs with valency $2$ are cycles, so $X$ is an
$n$-cycle.

(ii) Assume $2\le i<d$ and $A_1A_i=A_{i-1}+A_i+A_{i+1}$.
Then $b_{i-1}=1, a_i=1, c_{i+1}=1$.
Again, monotonicity gives $b_i=1$ and $c_i=1$, so $b_0=b_i+a_i+c_i=1+1+1=3$.
Moreover, since the $b_j$'s are positive and non-increasing, $b_j=1$ for
$j=i-1,\dots,d-1$.  Since the $c_j$'s are positive and non-decreasing,
$c_k=1$ for $1\le k\le i+1$.

(iii) If $A_1A_d = A_{d-1} + A_d$, then
$b_{d-1}=1$ and $a_d=1$.
Since $b_d=0$, the row-sum identity at distance $d$ is $a_d+c_d=b_0$.
Hence $c_d=b_0-1$.
The standard monotonicity conditions on the intersection numbers give no
additional immediate restrictions from these equalities.  Of course, the
intersection array must still satisfy the usual feasibility conditions for
distance-regular graphs.  There are DRGs with these conditions, see the examples
that follow.

(iv) In a distance-regular graph we have
$A_1A_d=b_{d-1}A_{d-1}+a_dA_d$,
since $A_{d+1}=0$.  Therefore $A_1A_d=A_{d-1}$
holds if and only if $b_{d-1}=1$ and $a_d=0$.
At distance $d$, we have $b_d=0$, and the row-sum identity for the
intersection numbers is $b_d+a_d+c_d=b_0$.
Hence $a_d+c_d=b_0$.
Since $a_d=0$, it follows that $c_d=b_0$.

(Note that bipartiteness alone is not sufficient for case (iv): bipartiteness gives
$a_d=0$, but the additional condition $b_{d-1}=1$ is also required.
For example, the complete bipartite graph $K_{m,m}$ with $m>2$ is
distance-regular and bipartite with diameter $2$, but $A_1A_2=(m-1)A_1$,
so it does not satisfy $A_1A_2=A_1$.)
\end{proof}

\begin{exa}
We find one intersection array of DRGs satisfying case (ii) of the theorem in \cite{BCN}, it is for the $L_2(17)$ graph, of order $102$ and diameter $d=7$, with intersection array 
$$ [3,2,2,2,1,1,1;1,1,1,1,1,1,3].$$
It satisfies (ii) for exactly one $i$:  $A_1A_5 = A_4 + A_5 + A_6$. 

For DRGs satisfying the condition of case (iii), we find two small examples in \cite{BCN}: the $5$-cycle from the previous section, and the Coxeter graph of order $28$ and diameter $d=4$, with intersection array
$$ [3,2,2,1;1,1,1,2].$$
\end{exa}

\begin{exa} 
For examples of MPFs in association schemes that do not arise from a DRG, symmetric cyclotomic schemes are a good source.  These schemes arise as fusions of the thin schemes of an abelian group on the orbits of a subgroup of its automorphism group that contains the inversion automorphism.  For example, consider the cyclotomic scheme formed by the fusion of $C_{17} = \langle g \rangle$ on the orbits of the subgroup of $\aut (C_{17})$ of order $4$.  This subgroup is generated by the automorphism $g \mapsto g^4$, whose square is inversion.  Doing our calculations in the group algebra $\mathbb{Q}C_{17}$, we see that both 
$$ (g+g^4+g^{-1}+g^{-4})(g^2+g^8+g^{-2}+g^{-8}) = (g^3+g^{-5}+g^{-3}+g^{5})(g^6+g^7+g^{-6}+g^{-7}) = \sum_{i=1}^{16} g^i.$$
So these will correspond to MPFs in this symmetric scheme.  (In this specific example the MPFs also correspond to near-factorizations but that is not typically the case.) 

In the classification of small association schemes, we find more MPFs of this type.  The three nonisomorphic pseudocyclic association schemes of order $25$ with six basis elements of valency $4$ and the pseudocyclic association scheme of order $29$ with seven elements of valency $4$ are symmetric cyclotomic schemes admitting MPFs similar to the above, with the product of two adjacency matrices of valency $4$ being equal to the sum of four distinct basis elements.  

For a product in a symmetric cyclotomic scheme to give a loopless MPF, it is not enough that the product be a sum of distinct basic adjacency matrices: the identity relation must not occur.  Equivalently, the expansion must have coefficient $0$ on $A_0$, and coefficients $0$ or $1$ on all nonidentity basic adjacency matrices.  In symmetric fusions of cyclic groups of odd prime order, this phenomenon happens frequently.


\end{exa}

\section{MPFs in non-symmetric homogeneous coherent configurations}\label{sec:hom}

A \defin{homogeneous} coherent configuration on $v$ vertices satisfies all the defining properties of an association scheme except commutativity.  Let $\{I=A_0,A_1,\dots,A_d\}$ be the adjacency matrices of the relations in a homogeneous coherent configuration.  When the configuration is not commutative, so not an association scheme, some of the adjacency matrices will be non-symmetric.  For all $i \in \{1,\dots,d\}$, let $i' \in \{1,\dots,d\}$ be the index for which $A_i^{\top} = A_{i'}$.  If $S \subseteq \{1,\dots,d\}$, in order for $A_S$ to be a symmetric matrix, it must be the case that $i \in S \iff i' \in S$.  

Let $\mathcal{A}\subseteq M_v(\mathbb{C})$ be the adjacency algebra of a homogeneous coherent configuration.  $\mathcal{A}$ is a finite‐dimensional semisimple $*$–algebra (with $*$ the conjugate transpose on $M_v(\mathbb{C})$). By Wedderburn, there exist integers $m_\rho,n_\rho\ge1$ and a $*$–isomorphism
\[
\Phi:\ \mathcal{A}\ \xrightarrow{\ \cong\ }\ \bigoplus_{\rho=0}^{r}\Bigl(I_{m_\rho}\otimes M_{n_\rho}(\mathbb{C})\Bigr),
\]
sending $X\in\mathcal{A}$ to the block tuple $\Phi(X)=(X_\rho)_{\rho=0}^r$ with $X_\rho\in M_{n_\rho}(\mathbb{C})$ and
\[
X\ =\ \bigoplus_{\rho=0}^{r}\bigl(I_{m_\rho}\otimes X_\rho\bigr),\qquad
\operatorname{tr}(XY)\ =\ \sum_{\rho=0}^{r} m_\rho\,\operatorname{tr}(X_\rho Y_\rho).
\]
The block $\rho=0$ is the \emph{principal} block.  For every union of
non-identity basic relations
\[
A_S=\sum_{i\in S}A_i,\qquad S\subseteq\{1,\dots,d\},
\]
we write
\[
k(S)=\sum_{i\in S}k_i.
\]
Then $A_S\mathbf{1}=k(S)\mathbf{1}$.  In a homogeneous coherent
configuration the column sum of $A_S$ is also $k(S)$, since the transpose
of each basic relation is again a basic relation with the same valency.

In this section, an MPF in a homogeneous coherent configuration is understood in the relational sense: $A_SA_T=A_U$, where
$S,T,U\subseteq\{1,\dots,d\}$, the product is a $0$-$1$ union of
non-identity basic relations, and no symmetry of $A_S,A_T,A_U$ is required.
Thus this is a directed version of the graph MPF considered in symmetric
association schemes.

\begin{lem}\label{lem:G2-valency}
Let $S,T\subseteq\{1,\dots,d\}$ be nonempty.  If $A_SA_T=A_U$
is a loopless MPF in the adjacency algebra of a homogeneous coherent
configuration on $v$ vertices, then
$1\le k(S),k(T)$, and  $k(S)+k(T)\le v-1$,
and $k(U)=k(S)k(T)\le v-1$.
\end{lem}

\begin{proof}
Since $S$ and $T$ are nonempty unions of non-identity basic relations, their
valencies are positive.  
Hence $1\le k(S),k(T)$.
Next, applying the identity $A_SA_T=A_U$ to the all-ones vector gives
\[
A_U\mathbf{1}
=
A_SA_T\mathbf{1}
=
A_S\bigl(k(T)\mathbf{1}\bigr)
=
k(T)k(S)\mathbf{1}.
\]
Therefore $k(U)=k(S)k(T)$.
Since $A_U$ is loopless, each row of $A_U$ has ones only off the diagonal,
so $k(U)\le v-1$.

It remains to prove the degree-sum bound.  Fix a vertex $x$.  Let
\[
N_S^+(x)=\{z\in X:(x,z)\in \bigcup_{i\in S}R_i\}
\textit{ and }
N_T^-(x)=\{z\in X:(z,x)\in \bigcup_{j\in T}R_j\}.
\]
Because $S$ and $T$ avoid the identity relation, both sets are contained in
$X\sm \{x\}$.  Moreover,
$|N_S^+(x)|=k(S)$, and  $|N_T^-(x)|=k(T)$,
using the row sum of $A_S$ and the column sum of $A_T$.
If $z\in N_S^+(x)\cap N_T^-(x)$, then there is a two-step walk $x \xrightarrow{S} z \xrightarrow{T} x$,
so $(A_SA_T)_{xx}>0$.  This contradicts the loopless condition, since
$A_SA_T=A_U$ and $A_U$ has zero diagonal.  Hence $N_S^+(x)\cap N_T^-(x)=\emptyset$.
Thus $N_S^+(x)$ and $N_T^-(x)$ are disjoint subsets of $X\sm \{x\}$,
and therefore
\[
k(S)+k(T)
=
|N_S^+(x)|+|N_T^-(x)|
\le v-1.
\]
This proves the claim.
\end{proof}

\begin{exa}
Matrix product factorizations do occur in non-commutative homogeneous configurations that are not thin.  The smallest one has order $10$, it is a Schurian configuration whose adjacency algebra corresponds to the double coset algebra of the group $C_5 \rtimes C_4$ of order $20$ with respect to a non-normal subgroup of order $2$.  The homogeneous coherent configuration has two elements of valency $1$, two symmetric elements of valency $2$ and two non-symmetric elements of valency $2$.  In this example, the product of the two symmetric elements of valency $2$ is their sum.  This example is perhaps not so interesting because this configuration admits a commutative $3$-class fusion which isolates these two elements, so this MPF 
is actually realized by two adjacency matrices of an association scheme. 

For an example of an MPF in a homogeneous coherent configuration that is not realized in an association scheme, we can consider the Schurian configuration corresponding to the action of $Sym(4)$ on the cosets of a subgroup $H$ of order $2$.  This gives a homogeneous configuration of order $12$, with two elements of valency $1$, and five elements of valency $2$, call these $A_2, \dots, A_6$.  The first three elements of valency $2$ correspond to double cosets represented by elements of valency $2$; these are symmetric.  The other two, $A_5$ and $A_6$, are represented by elements of order $3$ and $4$ only, and the inverse permutes this pair of double cosets, so they are not symmetric.  In the adjacency algebra of this configuration, the adjacency matrices of the symmetric elements do not commute.  We find 
$$ A_2 A_3 = A_4 + A_5, \quad A_3 A_2 = A_4 + A_6, $$ 
$$ A_2 A_4 = A_3 + A_6, \quad A_4 A_2 = A_3 + A_5, $$
and 
$$ A_3 A_4 = A_2 + A_5, \quad A_4 A_3 = A_2 + A_6. $$
So this gives an example of an MPF in a homogeneous coherent configuration that cannot be realized in an association scheme.     
\end{exa}

\section{Rank bounds for MPFs in symmetric association schemes}\label{sec:rankbounds}

\begin{thm}\label{thm:rank-mpf}
Let $\mathcal{R}$ be a symmetric $d$-class association scheme on $v$ vertices
with adjacency matrices $A_0=I,A_1,\dots,A_d$.  Let
$S,T,U\subseteq\{1,\dots,d\}$, and suppose
\[
A_SA_T=A_U
\]
is a loopless matrix-product factorization in the scheme.  Then:
\begin{enumerate}[(i)]
\item $\rank(A_U)\le \min\{\rank(A_S),\rank(A_T)\}$.

\item If $U\ne\emptyset$, then $\rank(A_U)\ge \frac{v}{k(U)}$.
Consequently, $\rank(A_S),\ \rank(A_T)\ge \frac{v}{k(U)}$.
Equivalently, if
\[
A_S=\sum_{h=0}^d \lambda_h(S)E_h,
\qquad
A_T=\sum_{h=0}^d \lambda_h(T)E_h,
\]
and $m_h=\tr(E_h)$, then
\[
\sum_{\{h:\lambda_h(S)\ne0\}}m_h\ge \frac{v}{k(U)},
\qquad
\sum_{\{h:\lambda_h(T)\ne0\}}m_h\ge \frac{v}{k(U)}.
\]
In particular, when $U=\{i\}$ is a single relation, this gives
\[
\rank(A_S),\ \rank(A_T)\ge \frac{v}{k_i},
\]
where $k_i$ is the valency of $A_i$.
\end{enumerate}
\end{thm}

\begin{proof}
For part (i), we use the elementary rank inequality
\[
\rank(MN)\le \min\{\rank(M),\rank(N)\}
\]
for matrices $M,N$ of compatible sizes.  Applying this to
$M=A_S$ and $N=A_T$, and using $A_SA_T=A_U$, gives
\[
\rank(A_U)=\rank(A_SA_T)\le \min\{\rank(A_S),\rank(A_T)\}.
\]

For part (ii), let $W\subseteq\{1,\dots,d\}$ be nonempty.  Since the scheme is
symmetric, $A_W$ is the adjacency matrix of a loopless $k(W)$-regular graph
on $v$ vertices.  Hence $\tr(A_W^2)=vk(W)$.
Let the eigenvalues of $A_W$ be $\theta_1,\dots,\theta_v$, and let
$r=\rank(A_W)$.  Since $A_W$ is the adjacency matrix of a
$k(W)$-regular graph, every eigenvalue satisfies $|\theta_j|\le k(W)$.
Therefore
\[
vk(W)
=
\tr(A_W^2)
=
\sum_{j=1}^v \theta_j^2
=
\sum_{\theta_j\ne0}\theta_j^2
\le
r\,k(W)^2.
\]
Since $k(W)>0$, we obtain $\rank(A_W)=r\ge \frac{v}{k(W)}$.

Now apply this with $W=U$.  Since $U\ne\emptyset$, this gives
$\rank(A_U)\ge \frac{v}{k(U)}$.
Combining this with part (i), we obtain
\[
\rank(A_S),\ \rank(A_T)\ge \rank(A_U)\ge \frac{v}{k(U)}.
\]

Finally, in a symmetric association scheme the adjacency matrices are
simultaneously diagonalizable.  If
\[
A_S=\sum_{h=0}^d \lambda_h(S)E_h,
\]
then $A_S$ acts as multiplication by $\lambda_h(S)$ on the image of
$E_h$, whose dimension is $m_h=\tr(E_h)$.  Therefore
\[
\rank(A_S)=\sum_{\{h:\lambda_h(S)\ne0\}}m_h.
\]
The same argument applies to $A_T$.  This gives the stated spectral form.
When $U=\{i\}$, we have $k(U)=k_i$, giving the final special case.
\end{proof}

\begin{thm}\label{thm:DRG-extremal-rank}
Let $\mathcal{R}$ be a symmetric $d$-class association scheme on $v$ vertices
with adjacency matrices $A_0=I,A_1,\dots,A_d$.  Let
$\emptyset\ne U\subseteq\{1,\dots,d\}$, and put $k=k(U)$.  Suppose $\rank(A_U)=\frac{v}{k}$.
Then every non-zero eigenvalue of $A_U$ is equal to either $k$ or $-k$; equivalently,
$\spec(A_U)\subseteq\{-k,0,k\}$.
Moreover, both $k$ and $-k$ occur as eigenvalues, the graph $\Gamma_U$ with
adjacency matrix $A_U$ is bipartite, and
$A_U^3=k^2A_U$.
In particular, the eigenvalue $0$ occurs if and only if $\rank(A_U)<v$.
\end{thm}

\begin{proof}
Let $A=A_U$.  Since $A$ is the adjacency matrix of a loopless $k$-regular graph
on $v$ vertices, we have $\tr(A^2)=vk$.
Let the eigenvalues of $A$ be $\theta_1,\dots,\theta_v$, and let
$r=\rank(A)$ be the number of non-zero eigenvalues.  Since $A$ is the adjacency
matrix of a $k$-regular graph, every eigenvalue satisfies $|\theta_j|\le k$.
Therefore
\[
vk=\tr(A^2)=\sum_{j=1}^v \theta_j^2
          =\sum_{\theta_j\ne 0}\theta_j^2
          \le r k^2.
\]
By hypothesis, $r=\rank(A)=v/k$, so the inequality is in fact an equality:
\[
vk=\sum_{\theta_j\ne 0}\theta_j^2\le \frac{v}{k}k^2=vk.
\]
Hence equality holds term by term, and every non-zero eigenvalue satisfies $\theta_j^2=k^2$.
Thus every non-zero eigenvalue is either $k$ or $-k$, so $\spec(A_U)\subseteq\{-k,0,k\}$.
The eigenvalue $k$ occurs because $A\mathbf{1}=k\mathbf{1}$.  Since $A$ has zero
diagonal, $\tr(A)=0$.  If $m_+$ and $m_-$ denote the multiplicities of $k$ and
$-k$, respectively, then
$0=\tr(A)=m_+k-m_-k$,
so $m_+=m_-$.  
In particular, $-k$ also occurs.
The spectrum is therefore symmetric about $0$.  Hence $\Gamma_U$ is bipartite.
Finally, since every eigenvalue of $A$ lies in $\{-k,0,k\}$, the minimal polynomial of $A$ divides $x(x^2-k^2)$.
Consequently $A^3=k^2A$.
The eigenvalue $0$ occurs exactly when $A$ is not full rank, i.e. exactly when
$\rank(A_U)<v$.
\end{proof}

\subsection{Hamming schemes \texorpdfstring{$H(d,q)$}{H(d,q)}}\label{sec:hamm}

The Hamming scheme $H(d,q)$ has vertex set $X=\{0,1,\dots,q-1\}^d$, with relations
\[
R_i = \{(x,y)\in X\times X : d_H(x,y)=i\},\qquad 0\le i\le d,
\]
and distance–$i$ adjacency matrices $A_i$ defined by
\[
(A_i)_{xy}=1\iff d_H(x,y)=i, \, i=0,1,\dots,d.
\]

\begin{prop}\label{prop:Hamming-valency}
For $0\le i\le d$, the valency of $A_i$ in $H(d,q)$ is
\[
k_i=(q-1)^i\binom{d}{i}.
\]
\end{prop}

\begin{proof}
Fix $x\in X$. To obtain $y$ at Hamming distance $i$ from $x$, choose the $i$ coordinates where $x$ and $y$ differ (there are $\binom{d}{i}$ choices) and, in each such coordinate, choose one of the $q-1$ values different from $x$'s value. Hence $k_i=(q-1)^i\binom{d}{i}$.
\end{proof}

\begin{prop}\label{prop:Hamming-rank}
Let $H(d,q)$ be the Hamming scheme on $v=q^d$ vertices. Suppose
\[
A_SA_T=A_u
\]
is a matrix–product factorization with $S,T\subseteq\{1,\dots,d\}$ and $U=\{u\}$ a single distance relation. Then
\[
\rank (A_S),\ \rank (A_T)\ \ge\ \frac{v}{k_u}
=\frac{q^d}{(q-1)^u\binom{d}{u}}.
\]
\end{prop}

\begin{proof}
This is the specialization of~\Cref{thm:rank-mpf}(ii) to $v=q^d$ and $k_u$ as in \Cref{prop:Hamming-valency}.
\end{proof}

We now sharpen the structure of MPFs of the form $A_1A_S$ in the binary case.

\begin{thm}\label{thm:Hn2-left-factor}
Let $H(d,2)$ be the binary Hamming scheme with distance matrices
$A_0,\dots,A_d$, where $d\ge1$.  Let
$T\subseteq\{1,\dots,d\}$, and suppose $A_1A_T=A_U$
is a loopless matrix-product factorization with
$U\subseteq\{1,\dots,d\}$.  Then the following hold:
\begin{enumerate}[(i)]
\item If $d=1$, then $T=\emptyset$ and $U=\emptyset$.
\item If $d\ge2$, then either $T=\emptyset$ and $U=\emptyset$, or $T=\{d\}, U=\{d-1\}$,
in which case $A_1A_d=A_{d-1}$.
\end{enumerate}
Thus, for $d\ge2$, the only nonzero loopless MPF of this form is
$A_1A_d=A_{d-1}$.  This MPF is trivial, since $A_d$ has valency $1$.
\end{thm}

\begin{proof}
The $d$-cube $Q_d$ underlying $H(d,2)$ is distance-regular with
intersection array
\[
[d,d-1,\dots,1;\,1,2,\dots,d].
\]
Hence, for $1\le i\le d$,
\[
A_1A_i
=
b_{i-1}A_{i-1}+a_iA_i+c_{i+1}A_{i+1}
=
(d-i+1)A_{i-1}+(i+1)A_{i+1},
\]
where $A_{d+1}=0$, and $a_i=0$ because the hypercube is bipartite.

Let $T\subseteq\{1,\dots,d\}$.  If $i\in T$ with $i<d$, then the
coefficient of $A_{i+1}$ in $A_1A_i$ is $i+1\ge2$.
Since all coefficients in the expansion of $A_1A_T$ are nonnegative, this
would force a coefficient larger than $1$, so $A_1A_T$ could not be a
$0$-$1$ adjacency matrix.  Therefore no element $i<d$ can belong to
$T$.

Thus either $T=\emptyset$ or $T=\{d\}$.  If $T=\emptyset$, then $A_1A_T=0=A_{\emptyset}$,
so $U=\emptyset$.
Now suppose $T=\{d\}$.  Then
$A_1A_d=A_{d-1}$.
If $d=1$, this reads
$A_1A_1=A_0$,
which is not loopless.  Hence $T=\{d\}$ is impossible when $d=1$.
If $d\ge2$, then $d-1\ge1$, so $A_{d-1}$ is loopless and the product is
a valid MPF.
Finally, the valency of $A_d$ in $H(d,2)$ is
\[
k_d=(2-1)^d\binom{d}{d}=1.
\]
Therefore the nonzero example $A_1A_d=A_{d-1}$ is trivial according to our
definition.
\end{proof}

\begin{rem}
In the binary Hamming scheme $H(d,2)$, the relation $A_d$ has valency
\[
k_d=(2-1)^d\binom{d}{d}=1.
\]
Therefore the factorization $A_1A_d=A_{d-1}$
is trivial according to our definition, since one of the factors has valency $1$.
Thus this result should not be described as a nontrivial MPF. More precisely, for
$d\ge 2$, the only non-zero loopless MPF of the form $A_1A_T=A_U$ is $T=\{d\}$, and $U=\{d-1\}$,
and this is a trivial MPF.
\end{rem}


We next record the corresponding non-binary obstruction.  Here the empty
choice $T=\emptyset$ always gives the zero identity $A_1A_{\emptyset}=0=A_{\emptyset}$,
so the correct statement is that there is no non-zero loopless MPF of this form.

\begin{prop}\label{prop:nonbinary-Hamming-left-factor}
Let $H(d,q)$ be a Hamming scheme with $q>2$. If
$T\subseteq\{1,\dots,d\}$ and $A_1A_T=A_U$
is a loopless matrix-product factorization, then $T=\emptyset$ and
$U=\emptyset$.  Equivalently, there is no non-zero loopless MPF of the form
$A_1A_T=A_U$ in $H(d,q)$ when $q>2$.
\end{prop}

\begin{proof}
The Hamming graph $H(d,q)$ is distance-regular with intersection numbers $b_i=(d-i)(q-1)$, $c_i=i$, $a_i=i(q-2)$.
Hence, for $1\le i\le d$,
$A_1A_i
=
b_{i-1}A_{i-1}+a_iA_i+c_{i+1}A_{i+1}$,
where $A_{d+1}=0$.  If $1\le i<d$, then the coefficient of
$A_{i+1}$ is $c_{i+1}=i+1\ge 2$.
Therefore no such $i$ can belong to $T$, since otherwise $A_1A_T$ would
have a coefficient larger than $1$.
It remains only to consider $i=d$.  In this case $A_1A_d=b_{d-1}A_{d-1}+a_dA_d
=(q-1)A_{d-1}+d(q-2)A_d$.
Since $q>2$, both $q-1\ge2$ and $d(q-2)\ge1$, so this product is not a
$0$-$1$ adjacency matrix of a loopless graph.  Hence $d\notin T$ as well.
Thus $T=\emptyset$, and consequently $A_1A_T=0=A_{\emptyset}$, so
$U=\emptyset$.
\end{proof}




\begin{exa}\label{exa:H32}
In $H(3,2)$ we have distance matrices $A_0,A_1,A_2,A_3$. Using the intersection array
\[
[3,2,1;\,1,2,3],
\]
we obtain $A_1A_1 = 3A_0+2A_2$,
$A_1A_2 = 2A_1+3A_3$,
$A_1A_3 = A_2$.
None of $A_1A_1$ or $A_1A_2$ is an MPF (coefficients $>1$), while $A_1A_3=A_2$ is a special case of~\Cref{thm:Hn2-left-factor} with $d=3$.
\end{exa}

\section{Translation and cyclotomic schemes}\label{sec:cyclotomic}

In this section we explain the connection between MPFs in translation
association schemes and near-factorizations in finite abelian groups.  This
also clarifies the overlap with Cayley graph MPFs studied in earlier work and
with recent cyclotomic constructions of $\lambda$-fold near-factorizations.

Let $G$ be a finite abelian group, written multiplicatively, with identity
element $e$.  For a subset $S\subseteq G$, write
\[
   \underline{S}:=\sum_{s\in S}s\in \mathbb ZG
\]
for the corresponding group-ring element.  If $S\subseteq G\sm\{e\}$,
let $A_S$ denote the adjacency matrix of the Cayley graph
$\operatorname{Cay}(G,S)$, so that
\[
   (A_S)_{x,y}=1
   \quad\Longleftrightarrow\quad
   x^{-1}y\in S .
\]

Let $K\leq \operatorname{Aut}(G)$ contain the inversion automorphism
$g\mapsto g^{-1}$.  The orbits of $K$ on $G\sm\{e\}$ define a symmetric
fusion of the thin translation scheme of $G$.  Thus whenever $S$ is a union of
$K$-orbits, $A_S$ is a graph in this symmetric association scheme.

\begin{prop}\label{prop:translation-group-ring}
Let $S,T\subseteq G\sm\{e\}$.  For every $x,y\in G$, the entry
$(A_SA_T)_{x,y}$ is the coefficient of $x^{-1}y$ in the group-ring product
$\underline{S}\,\underline{T}$.  Consequently, if
$U\subseteq G\sm\{e\}$, then
\[
   A_SA_T=A_U
   \qquad\Longleftrightarrow\qquad
   \underline{S}\,\underline{T}=\underline{U}
   \quad\text{in } \mathbb ZG .
\]
In particular, if $S,T,U$ are unions of $K$-orbits, then a group-ring identity
$\underline{S}\,\underline{T}=\underline{U}$ is precisely an MPF in the
corresponding symmetric translation association scheme.
\end{prop}

\begin{proof}
The number $(A_SA_T)_{x,y}$ counts the vertices $z\in G$ such that
$x^{-1}z\in S$ and $z^{-1}y\in T$.  Writing
$s=x^{-1}z$ and $t=z^{-1}y$, this is equivalent to $st=x^{-1}y$.
Thus $(A_SA_T)_{x,y}$ is exactly the number of representations of $x^{-1}y$
as a product $st$ with $s\in S$ and $t\in T$, which is the coefficient of
$x^{-1}y$ in $\underline{S}\,\underline{T}$.  The stated equivalence follows immediately.
\end{proof}

\begin{cor}\label{cor:cayley-overlap}
Every abelian Cayley graph MPF with inverse-closed connection sets is realized
inside a symmetric association scheme.  More precisely, if
$S,T,U\subseteq G\sm\{e\}$ are inverse-closed and
$A_SA_T=A_U$
as Cayley graph adjacency matrices, then the same identity is an MPF in the
symmetric orbit fusion of the thin translation scheme obtained from the orbits
of $\langle g\mapsto g^{-1}\rangle$.
\end{cor}

\begin{proof}
If $S,T,U$ are inverse-closed, then they are unions of the orbits of the
inversion automorphism.  The result follows from
\Cref{prop:translation-group-ring}.
\end{proof}

\begin{rem}\label{rem:nf-vs-mpf}
A near-factorization of $G$ is the special case
$\underline{S}\,\underline{T}=\underline{G\sm\{e\}}$.
Equivalently, in the associated translation scheme, $A_SA_T=J-I$.
More generally, a $\lambda$-fold near-factorization satisfies
$\underline{S}\,\underline{T}
   =
   \lambda\,\underline{G\sm\{e\}}$,
and hence gives $A_SA_T=\lambda(J-I)$.
Thus a $\lambda$-fold near-factorization is an MPF in the present sense exactly
when $\lambda=1$.  For $\lambda>1$ it gives a natural weighted analogue, but
not a $0$-$1$ matrix product factorization.
\end{rem}

The next construction gives an infinite family of cyclotomic MPFs which are
not near-factorizations.

\begin{thm}\label{thm:projective-line-mpf}
Let $V=\mathbb F_q^2$ be regarded as an additive group, and let
$\mathfrak X_q$ be the translation association scheme whose nontrivial classes
are the punctured $1$-dimensional subspaces $L^\#:=L\sm\{0\}$, where $L\in \operatorname{PG}(1,q)$.
Equivalently, $\mathfrak X_q$ is the orbit scheme for the scalar action of
$\mathbb F_q^\times$ on $V\sm\{0\}$.
If $L,M$ are distinct $1$-dimensional subspaces of $V$, then
\[
   A_{L^\#}A_{M^\#}
   =
   A_{V\sm(L\cup M)}
   =
   \sum_{\substack{N\in\operatorname{PG}(1,q)\\ N\ne L,M}} A_{N^\#}.
\]
Consequently, for $q>2$, this gives a nontrivial MPF in a symmetric cyclotomic
association scheme.  It is not a near-factorization, since the output is a
proper subset of $V\sm\{0\}$.
\end{thm}

\begin{proof}
Since $L$ and $M$ are distinct $1$-dimensional subspaces of the $2$-dimensional
space $V$, we have $V=L\oplus M$.
Thus every $v\in V$ has a unique expression $v=\ell+m$, $\ell\in L,\quad m\in M$.
This expression has $\ell\ne0$ and $m\ne0$ if and only if
$v\notin L\cup M$.  Therefore every element of
$V\sm(L\cup M)$ occurs exactly once in $\underline{L^\#}\,\underline{M^\#}$,
and every element of $L\cup M$ occurs zero times.  Hence $\underline{L^\#}\,\underline{M^\#}
   =
   \underline{V\sm(L\cup M)}$.
The matrix identity follows from
\Cref{prop:translation-group-ring}.  The factor valencies are
$q-1$, while the output valency is
\[
   |V\sm(L\cup M)|
   =
   q^2-|L\cup M|
   =
   q^2-(2q-1)
   =
   (q-1)^2,
\]
as required by valency multiplicativity.
\end{proof}

\begin{rem}
Identifying $V=\mathbb F_q^2$ with $\mathbb F_{q^2}$, the punctured
$1$-dimensional $\mathbb F_q$-subspaces are precisely the cosets of
$\mathbb F_q^\times$ in $\mathbb F_{q^2}^\times$.  Thus~\Cref{thm:projective-line-mpf} is a cyclotomic construction of order
$q+1$.  For $q=5$, this recovers an explicit MPF in a $6$-class cyclotomic
scheme on $25$ vertices, with two factors of valency $4$ and output a union of
four basic relations.
\end{rem}

We now record a cyclotomic-number criterion for MPFs in finite-field
translation schemes.  This gives a direct way to compare the present MPF
condition with cyclotomic constructions of near-factorizations.
Let $q$ be an odd prime power, let $\gamma$ be a primitive element of
$\mathbb F_q$, and let $N\mid(q-1)$.  Put $C_i\coloneqq \gamma^i\langle \gamma^N\rangle$, where $i\in\mathbb Z/N\mathbb Z$.
Assume that $-1\in C_0$, equivalently $(q-1)/N$ is even, so that the
corresponding cyclotomic translation scheme is symmetric.  Let $B_i$ be the
adjacency matrix of the relation
\[
   (x,y)\in R_i
   \quad\Longleftrightarrow\quad
   y-x\in C_i .
\]
For $\mathcal I\subseteq\mathbb Z/N\mathbb Z$, we define the following 
\[
   C_{\mathcal I}:=\bigcup_{i\in\mathcal I}C_i,
   \qquad
   B_{\mathcal I}:=\sum_{i\in\mathcal I}B_i .
\]

\begin{nota}
Let $(a,b)_N\coloneqq |(C_a+1)\cap C_b|$
denote the cyclotomic numbers of order $N$, with indices taken modulo $N$.
For $\mathcal I,\mathcal J\subseteq\mathbb Z/N\mathbb Z$ and
$h\in\mathbb Z/N\mathbb Z$, set
\[
   m_h(\mathcal I,\mathcal J)
   :=
   \sum_{i\in\mathcal I}\sum_{j\in\mathcal J}
   (i-h,j-h)_N .
\]  
\end{nota}

\begin{thm}\label{thm:cyclotomic-criterion}
There is a loopless MPF
$B_{\mathcal I}B_{\mathcal J}=B_{\mathcal L}$
if and only if
$\mathcal I\cap\mathcal J=\varnothing$
and
$m_h(\mathcal I,\mathcal J)\in\{0,1\}$ for every $h\in\mathbb Z/N\mathbb Z$.
In that case $\mathcal L
   =
   \{h\in\mathbb Z/N\mathbb Z:
      m_h(\mathcal I,\mathcal J)=1\}$.
Moreover, $B_{\mathcal I}B_{\mathcal J}
   =
   \lambda(J-I)$
if and only if
$\mathcal I\cap\mathcal J=\varnothing$ and $m_h(\mathcal I,\mathcal J)=\lambda$
for every $h$.
Thus the case $\lambda=1$ is exactly the near-factorization case, while
$\lambda>1$ is a weighted analogue rather than an MPF in the present sense.
\end{thm}

\begin{proof}
The coefficient of $0$ in
$\underline{C_{\mathcal I}}\underline{C_{\mathcal J}}$ is the number of pairs
$(x,y)\in C_{\mathcal I}\times C_{\mathcal J}$ with $x+y=0$.  Since
$-1\in C_0$, this coefficient is
\[
   |C_{\mathcal I}\cap C_{\mathcal J}|
   =
   \frac{q-1}{N}|\mathcal I\cap\mathcal J|.
\]
Thus the loopless condition is exactly
$\mathcal I\cap\mathcal J=\varnothing$.
Now fix a non-zero element $c\in C_h$.  The coefficient of $c$ in
$\underline{C_i}\underline{C_j}$ is the number of $x\in C_i$ such that
$c-x\in C_j$.  Dividing by $c$, this is the number of $u\in C_{i-h}$ such that $1-u\in C_{j-h}$.
Since $-1\in C_0$, replacing $u$ by $-u$ shows that this number is
\[
   |(C_{i-h}+1)\cap C_{j-h}|
   =
   (i-h,j-h)_N.
\]
Summing over $i\in\mathcal I$ and $j\in\mathcal J$ gives
$m_h(\mathcal I,\mathcal J)$.
Therefore the product is a $0$-$1$ union of cyclotomic relations exactly when
all the non-zero coefficients $m_h(\mathcal I,\mathcal J)$ are $0$ or $1$, and
the output classes are precisely those with coefficient $1$.  The final
statement follows in the same way, with all non-zero coefficients equal to
$\lambda$.
\end{proof}

\begin{exa}
Let $G=(\mathbb F_{29},+)$ and let $\gamma=2$, which is primitive in
$\mathbb F_{29}$.  Take $N=7$, so that each cyclotomic class has size $4$:
$C_i=2^i\langle 2^7\rangle$, where $i\in\mathbb Z/7\mathbb Z$.
Since $-1\in C_0$, the corresponding cyclotomic scheme is symmetric.  A direct
calculation in $\mathbb ZG$ gives
\[
   \underline{C_0}\,\underline{C_1}
   =
   \underline{C_0}
   +\underline{C_2}
   +\underline{C_5}
   +\underline{C_6}.
\]
Hence, in the cyclotomic association scheme on $29$ vertices,
\[
   B_0B_1=B_0+B_2+B_5+B_6.
\]
This is an MPF whose factors are basic relations of valency $4$.  It is not a
near-factorization of $G$, because the output is a proper union of four of the
seven nonidentity cyclotomic classes rather than all of $G\sm\{0\}$.
\end{exa}

\begin{rem}
\Cref{thm:cyclotomic-criterion} explains how the methods used in
cyclotomic constructions of $\lambda$-fold near-factorizations can be adapted
to the present problem.  A $\lambda$-fold near-factorization corresponds to the
constant condition $m_h(\mathcal I,\mathcal J)=\lambda$ for all $h$, whereas an MPF only requires the sharper $0$-$1$ condition
\[
   m_h(\mathcal I,\mathcal J)\in\{0,1\}.
\]
Thus cyclotomic-number calculations can be used to search for MPFs that are
not near-factorizations, namely cases where the set of indices $h$ with
$m_h=1$ is a proper subset of the nonidentity cyclotomic classes.
\end{rem}

\section{Restrictions on MPFs in symmetric association schemes}
\label{sec:universal-pentagon}

\begin{thm}\label{thm:universal-pentagon}
Let $\mathcal{R}$ be a symmetric $d$-class association scheme on $v$ vertices with adjacency matrices $A_0=I,A_1,\dots,A_d$.  Let $S,T\subseteq\{1,\dots,d\}$ be non-empty index sets with
\[
S\cap T=\varnothing,
\qquad
S\cup T = \{1,\dots,d\}.
\]
Assume that $A_SA_T \;=\; A_{S\cup T}$.
Then $v=5$, $k(S)=k(T)=2$, and the graphs with adjacency matrices $A_S$ and $A_T$ are complementary $5$-cycles.
Equivalently, the fusion of $\mathcal{R}$ with classes $\{0\},S,T$ is the $2$-class association scheme of the cycle $C_5$.
\end{thm}

\begin{proof}
Because $S\cap T=\varnothing$ and $S\cup T=\{1,\dots,d\}$, we have
\begin{equation}\label{eq:ST-complete}
A_S + A_T
 = \sum_{i\in S\cup T} A_i
 = \sum_{i=1}^d A_i
 = J - I.
\end{equation}
Thus $A_{S\cup T}=J-I$, so the MPF assumption says $A_SA_T = A_{S\cup T} = J - I$.

Since  we have
\begin{equation}\label{eqn:fusion}
A_S^2 = A_S(J-I-A_T) = k(S)J - A_S - (J-I) = k(S)I + (k(S)-2)A_S + (k(S)-1)A_T, 
\end{equation}
we can see that $\{I=A_0, A_S, A_T\}$ is the set of matrices of a $2$-class fusion scheme of $\mathcal{R}$. 
The graph with adjacency matrix $A_S$ is strongly regular with parameters $(v,k,\lambda,\mu)$, where $k=k(S)$ is its valency, and the parameters $\lambda$ and $\mu$ satisfy  
$A_S^2 = k I + \lambda A_S + \mu A_T$. 




Hence $\mu = k-1,$ and $\lambda = k-2$.
The parameters $(v,k,\lambda,\mu)$ of a strongly regular graph satisfy $(v-k-1)\mu = k(k-\lambda-1)$,
Substituting $\mu=k-1$, $\lambda=k-2$ gives
$(v-k-1)(k-1) = k$.

If $k=1$ then the left-hand side is $0$, while the right-hand side is $1$, impossible.
Thus $k\ge2$, and we may divide by $k-1$ to obtain
\[
v-k-1 = \frac{k}{k-1}.
\]
Hence $k-1$ divides $k$, so $k-1=1$ and therefore $k=2$. Then $\lambda = 0$, $\mu=1$, and $v-2-1 = 2/(1)$ gives $v=5$.
We can conclude that $A_S$ is the adjacency matrix of a strongly regular graph with parameters $(v,k,\lambda,\mu)=(5,2,0,1)$, i.e., a $5$-cycle.
$A_T$ is its complement, which is again a $5$-cycle.  These are exactly the graphs with adjacency matrices $A_S$ and $A_T$.
\end{proof}

\section*{Acknowledgment}

This research was conducted while the second author was visiting the University of Regina. 
This research was supported by NSERC.

\section*{Data Availability} No datasets were generated or analysed during the current study.
\section*{Declarations}
The authors declare no conflict of interest.

\bibliographystyle{plainurlnat}
\bibliography{MPF.bib}
\nocite{*}
\newpage
    
\end{document}